\documentclass{amsart}

\usepackage{amssymb,amsthm}

\usepackage{a4wide}

\usepackage{graphicx}
\usepackage{hyperref}
\usepackage{color}

\title{A new $q$-Selberg integral, Schur functions, and Young books}
\author{Jang Soo Kim and Soichi Okada}
\thanks{}
\keywords{}
\date{\today}

\newtheorem{thm}{Theorem}[section]
\newtheorem{lem}[thm]{Lemma}
\newtheorem{prop}[thm]{Proposition}

\theoremstyle{definition}

\theoremstyle{remark}
\newtheorem{remark}{Remark}[section]

\newcommand{\YB}{\mathrm{YB}}    
\newcommand{\maj}{\mathrm{maj}}  
\newcommand{\Par}{\mathrm{Par}}  
\newcommand{\A}{\mathcal{A}}     
\newcommand{\LL}{\mathcal{L}}    
\newcommand{\N}{\mathbb{N}}      
\newcommand{\Z}{\mathbb{Z}}      

\newcommand\Sym{\mathfrak{S}}
\newcommand\vectr{\boldsymbol{r}}
\newcommand\vects{\boldsymbol{s}}

\begin{document}

\begin{abstract}
  Recently, Kim and Oh expressed the Selberg integral in terms of the
  number of Young books which are a generalization of standard Young
  tableaux of shifted staircase shape. In this paper the generating
  function for Young books according to major index statistic is
  considered. It is shown that this generating function can be written
  as a Jackson integral which gives a new $q$-Selberg integral.  It is
  also shown that the new $q$-Selberg integral has an expression in
  terms of Schur functions.
\end{abstract}

\maketitle

\section{Introduction}

For a positive integer $n$, and complex numbers 
$\alpha$, $\beta$, $\gamma$ with $\mathrm{Re} (\alpha) > 0$, 
$\mathrm{Re} (\beta) > 0$, and $\mathrm{Re} (\gamma) >
- \min \{ 1/n, \mathrm{Re}(\alpha)/(n-1), \mathrm{Re}(\beta)/(n-1)\}$, 
the Selberg integral $S_n(\alpha,\beta,\gamma)$ is defined by 
\begin{align}
\label{eq:Selberg}
S_n(\alpha,\beta,\gamma)
 &= 
\int_0^1 \cdots \int_0^1
  \prod_{i=1}^n x_i^{\alpha-1} (1-x_i)^{\beta-1}
  \prod_{1 \le i < j \le n} |x_i-x_j|^{2\gamma} dx_1 \cdots dx_n
\\ \notag
 &= 
\prod_{j=1}^n
  \frac{ \Gamma(\alpha+(j-1)\gamma)
         \Gamma(\beta+(j-1)\gamma)\Gamma(1+j\gamma) }
       { \Gamma(\alpha+\beta+(n+j-2)\gamma)\Gamma(1+\gamma) }.
\end{align}
The Selberg integral is a generalization of Euler's beta integral, 
which is $S_1(\alpha,\beta,0)$. 
The formula above is due to Selberg \cite{Selberg1944}.
There are many generalizations of the Selberg integral, see \cite{Forrester}.
As a $q$-analogue of (\ref{eq:Selberg}), Askey \cite{Askey} 
conjectured that
\begin{multline}
\label{eq:Askey}
\int_0^1 \cdots \int_0^1
 \prod_{i=1}^n
  x_i^{\alpha-1}
  (q x_i ; q)_{\beta-1}
 \prod_{1 \le i < j \le n}
  x_i^{2k} \left( q^{1-k} \frac{x_j}{x_i} ; q \right)_{2k}
 d_q x_1 \cdots d_q x_n
\\
=
q^{\alpha k \binom{n}{2} + 2 k^2 \binom{n}{3}}
\prod_{j=1}^n
  \frac{ \Gamma_q(\alpha+(j-1) k)
         \Gamma_q(\beta+(j-1) k)
         \Gamma_q(1+jk)
       }
       { \Gamma_q(\alpha+\beta+(n+j-2)k)
         \Gamma_q(1+k)
       },
\end{multline}
where $(x;q)_s$ is the $q$-Pochhammer symbol and $\Gamma_q(x)$ is the $q$-Gamma function.
The Askey conjecture (\ref{eq:Askey}) has been proved by Habsieger
\cite{Habsieger} and Kadell \cite{Kadell2} independently.

When $\alpha-1=r$, $\beta-1=s$, and $2\gamma=m$ are nonnegative
integers, the Selberg integral $S_n(\alpha,\beta,\gamma)$ has a
combinatorial interpretation due to Stanley \cite[Chapter 1, Exercise 11]{EC1}. 
Kim and Oh \cite{KimOh}
introduced a combinatorial object called ``$(n,\vectr,\vects)$-Young
books'', where $\vectr = (r_1, \dots, r_m)$ and $\vects = (s_1, \dots,
s_m)$ are (weak) compositions of $r$ and $s$ respectively, and showed
that
\begin{multline}
\label{eq:KO}
\frac{1}{n!}
\int_0^1 \cdots \int_0^1
 \prod_{k=1}^m
  \left(
   \prod_{i=1}^n x_i^{r_k} (1 - x_i)^{s_k}
   \prod_{1 \le i < j \le n} |x_j - x_i|
  \right)
 d x_1 \cdots d x_n
\\
=
\frac{1}{ N! }
\prod_{k=1}^m
 \frac{ F(n+r_k+s_k) }
      { F(r_k) F(s_k) }
 | \YB(n;\vectr,\vects) |,
\end{multline}                 
where $N = m \binom{n}{2} + (r+s+1)n + \sum_{k=1}^m r_k s_k$, $F(l) =
\prod_{i=1}^{l-1} i!$, and $\YB(n;\vectr,\vects)$ is the set of
$(n,\vectr,\vects)$-Young books.  Notice that the left hand side of
\eqref{eq:KO} is equal to $\frac{1}{n!}S_n(\alpha,\beta,\gamma)$ with
$\alpha-1=r$, $\beta-1=s$, and $2\gamma=m$.  Their proof of
\eqref{eq:KO} uses Stanley's combinatorial interpretation for the
Selberg integral and Postnikov's result \cite{Postnikov2009} on the generating function for
the number of standard Young tableaux of shifted staircase shape.

In this paper we generalize \eqref{eq:KO} by considering 
a $q$-analogue for the Young books.

A Young book can be thought of as a linear extension of a poset. 
Thus a natural $q$-analogue of $|\YB(n,\vectr,\vects)|$ is 
the generating function for $(n,\vectr,\vects)$-Young books 
according to the major index statistic of linear extensions. 
Interestingly, this $q$-analogue is the $q$-Selberg integral 
which is obtained from \eqref{eq:KO} by simply replacing 
$(1-x)^s$ by $(qx;q)_s = \prod_{i=1}^s (1 - q^i x)$ and 
the Riemann integral by the Jackson integral.

\begin{thm}
\label{thm:qKO}
Let $n$ and $m$ be positive integers, 
and $\vectr = (r_1, \dots, r_m)$ and $\vects = (s_1, \dots, s_m)$ compositions of
 $r$ and $s$ respectively.
Then, for $0<q<1$, we have
\begin{multline}
\label{eq:qKO}
\frac{1}{n!}
\int_0^1 \cdots \int_0^1
 \prod_{k=1}^m 
   \left(
     \prod_{i=1}^n x_i^{r_k} (q x_i ; q)_{s_k}
     \prod_{1 \le i < j \le n} |x_j - x_i|
   \right)
 d_q x_1 \cdots d_q x_n
\\
=
q^{ (r+1) \binom{n}{2} + m \binom{n}{3} }
\frac{1}{ [N]_q! }
\prod_{k=1}^m
 \frac{ F_q(n+r_k+s_k) }
      { F_q(r_k) F_q(s_k) }
\sum_{B \in \YB(n;\vectr,\vects)} q^{\maj(B)},
\end{multline}
where $N = m \binom{n}{2} + (r+s+1)n + \sum_{k=1}^m r_k s_k$, and
$F_q(l) = \prod_{i=1}^{l-1} [i]_q!$.
\end{thm}

As a part of the proof of Theorem~\ref{thm:qKO}, 
we show that the $q$-Selberg integral in Theorem~\ref{thm:qKO} 
can be written in terms of the principal specialization of Schur functions.

\begin{thm}
\label{thm:qSelberg-Schur}
For $0<q<1$, we have
\begin{multline}
\label{eq:qSelberg-Schur}
\frac{1}{n!}
\int_0^1 \cdots \int_0^1
 \prod_{k=1}^m 
   \left(
     \prod_{i=1}^n x_i^{r_k} (q x_i ; q)_{s_k}
     \prod_{1 \le i < j \le n} |x_j - x_i|
   \right)
 d_q x_1 \cdots d_q x_n
\\
=
(1-q)^n
q^{ (r+1) \binom{n}{2} + m \binom{n}{3} }
\prod_{k=1}^m
 \prod_{h=s_k}^{n+s_k-1} (q;q)_h
\sum_{\lambda \in \Par_n}
 q^{|\lambda|}
 \prod_{k=1}^m
  q^{r_k |\lambda|}
  s_\lambda(1, q, \dots, q^{n+s_k-1}),
\end{multline}
where $\Par_n$ is the set of all partitions of length $\le n$.
In particular, we have
\begin{multline}
\label{eq:qSelberg}
\frac{1}{n!}
\int_0^1 \cdots \int_0^1
 \prod_{i=1}^n x_i^r (q x_i ; q)_s
 \prod_{1 \le i < j \le n} |x_j - x_i|^m
 d_q x_1 \cdots d_q x_n
\\
=
(1-q)^n
q^{ (r+1) \binom{n}{2} + m \binom{n}{3} }
\prod_{h=s}^{n+s-1} (q;q)_h
\left(
 \prod_{h=1}^{n-1} (q;q)_h
\right)^{m-1}
\\
\times
\sum_{\lambda \in \Par_n}
 q^{(r+1)|\lambda|}
 s_\lambda(1, q, \dots, q^{n+s-1})
 s_\lambda(1, q, \dots, q^{n-1})^{m-1}.
\end{multline}
\end{thm}

This allows us to evaluate the $q$-Selberg integral in Theorem~\ref{thm:qKO} 
when $m=2$ and $\vects = (s,0)$ (or $(0,s)$) (resp. $m=1$ and $s = 0$ or $1$),
by using the Cauchy identity (resp. the Schur--Littlewood identity) for 
Schur functions.

The remainder of this paper is organized as follows. In
Section~\ref{sec:-generating-function} we define
$(n,\vectr,\vects)$-Young books. Considering $(n,\vectr,\vects)$-Young
books as linear extensions of a certain poset, we express the
generating function for $(n,\vectr,\vects)$-Young books according to
the major index in terms of Schur functions. In
Section~\ref{sec:-schur-functions} we prove Theorems~\ref{thm:qKO} and
\ref{thm:qSelberg-Schur} by expressing the generating function
obtained in Section~\ref{sec:-generating-function} as a
Jackson integral. In Section~\ref{sec:-evaluation-q} we evaluate
special cases of the $q$-Selberg integral in Theorem~\ref{thm:qKO}
using Cauchy and Schur-Littlewood identities.  In
Section~\ref{sec:-variants-q} we derive variants of the $q$-Selberg
integral \eqref{eq:Askey} using Cauchy-type identities for classical
group characters.

\section{%
The generating function of Young books
}
\label{sec:-generating-function}

In this section we define the Young books introduced in \cite{KimOh}
and consider their generating function using the major index statistic
of linear extensions of a poset.

Throughout this paper we denote $[n]=\{1,2,\dots,n\}$. 

\def\outline#1{{\bf \Large #1}\\}

\begin{figure}
  \centering
  \includegraphics{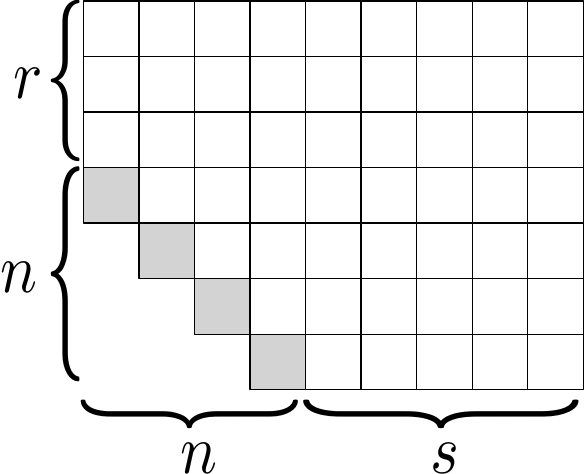}  
  \caption{An $(n,r,s)$-staircase with $n=4,r=3$ and $s=5$. The diagonal cells are shaded.}
  \label{fig:staircase}
\end{figure}
 
An \emph{$(n,r,s)$-staircase} is the diagram obtained from an
$(r+n)\times(n+s)$ rectangle by removing a staircase
$(n-1,n-2,\dots,1)$ from the lower-left corner, see
Figure~\ref{fig:staircase}. We label the rows of an
\emph{$(n,r,s)$-staircase} by $-r+1,-r+2,\dots,-1,0,1,2,\dots,n$ from
top to bottom and the columns by $1,2,\dots,n+s$ from left to
right. For $1\le i\le n$, the cell in the $i$th row and $i$th column
is called the \emph{$i$th diagonal cell}.

Let $\vectr=(r_1,r_2,\dots,r_m)$ and $\vects=(s_1,s_2,\dots,s_m)$ be
compositions of $r$ and $s$ respectively. In other words,
$r_1+r_2+\dots+r_m=r$, $s_1+s_2+\dots+s_m=s$ and $r_i,s_i\ge0$.

An \emph{$(n,\vectr,\vects)$-staircase} is an $m$-tuple
$(\lambda^{(1)}, \lambda^{(2)}, \dots, \lambda^{(m)})$ where each
$\lambda^{(k)}$ is an $(n,r_k,s_k)$-staircase and the $j$th diagonal
cells of $\lambda^{(1)}, \lambda^{(2)}, \dots, \lambda^{(m)}$ are
identified for $j=1,2,\dots,n$. We say that $\lambda^{(k)}$ is the
\emph{$k$th page} of this $(n,\vectr,\vects)$-staircase.  Note that an
\emph{$(n,\vectr,\vects)$-staircase} has $N$ cells, where
\[
N=m\binom n2+(r+s+1)n+\sum_{k=1}^m r_ks_k.
\]

An \emph{$(n,\vectr,\vects)$-Young book} is a filling of an
\emph{$(n,\vectr,\vects)$-staircase} with integers $1,2,\dots,N$ such
that in each page the entries are increasing from left to right in
each row and from top to bottom in each column.  See
Figure~\ref{fig:(n,r,s)-YB}. We denote by $\YB(n,\vectr,\vects)$ the
set of $(n,\vectr,\vects)$-Young books.

\begin{figure}
  \centering
  \includegraphics{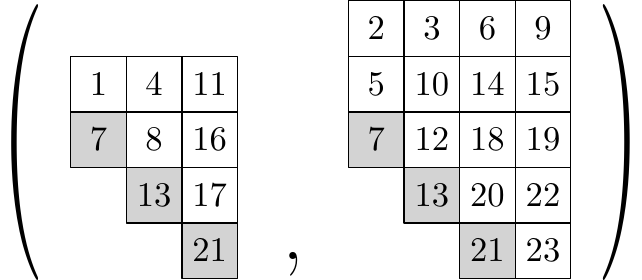}  
  \caption{An $(n,\vectr,\vects)$-Young book with $n=3,\vectr=(1,2)$,
    and $\vects=(0,1)$. The diagonal cells are shaded.}
  \label{fig:(n,r,s)-YB}
\end{figure}

 For $B\in \YB(n,\vectr,\vects)$, a
\emph{descent} of $B$ is an integer $i$ such that one of the following
conditions holds:
\begin{itemize}
\item the row containing $i+1$ has a smaller label than that of the
  row containing $i$ regardless of their page numbers,
\item $i$ and $i+1$ are non-diagonal entries and $i+1$ appears in an
  earlier page than $i$ with the same row label.
\end{itemize}
The \emph{major index} $\maj(B)$ of $B$ is
the sum of descents of $B$. For example, if $B$ is the Young book in
Figure~\ref{fig:(n,r,s)-YB}, then $\maj(B)=1+5+8+10+13+17+21=75$.

We can consider an $(n,\vectr,\vects)$-Young book as a linear
extension of a poset. We briefly introduce some terminologies. We
refer the reader to \cite[3.15]{EC1} for details.

Let $P$ by a poset with $p$ elements.  A \emph{linear extension} of
$P$ is an order-preserving bijection $\sigma:P \to [p]$, i.e., if $x
<_P y$ then $\sigma(x)<\sigma(y)$.  We denote by $\LL(P)$ the set of
linear extensions of $P$.  A \emph{$P$-partition} is an
order-reversing map $\sigma : P \to \N$, i.e., if $x <_P y$ then
$\sigma(x) \ge \sigma(y)$.  We denote by $\A(P)$ the set of all
$P$-partitions.

Note that, once we fix a linear extension $\omega$ of $P$, we can
consider $\pi\in \LL(P)$ as a permutation on $[p]$ by
$\omega\pi^{-1}$.  For a permutation $\pi = \pi_1 \pi_2 \dots
\pi_p$, a \emph{descent} of $\pi$ is an integer $i \in [p-1]$ such
that $\pi_i >\pi_{i+1}$. The \emph{major index} $\maj(\pi)$ is the sum
of descents of $\pi$.

It is well known, for example \cite[3.15.7 Theorem]{EC1}, that for a
poset $P$ with a fixed linear extension $\omega:P\to[p]$, we have
\begin{equation}
\label{eq:maj}
\sum_{\sigma \in \A(P)} q^{|\sigma|}
 =
\frac{ \sum_{\pi\in\LL(P) } q^{\maj(\omega\pi^{-1})} }
     { (q;q)_p },
\end{equation}
where $|\sigma| = \sum_{x \in P} \sigma(x)$.

Let $P_{n,\vectr,\vects}$ be the poset defined as follows:
\begin{itemize}
\item The underlying set consists of the cells in an
  $(n,\vectr,\vects)$-staircase.
\item The covering relation is defined by $c_1<c_2$ if $c_1$ and $c_2$
  are, respectively, in rows $i$ and $i'$ in columns $j$ and $j'$ in
  the same page such that $(i',j')=(i,j+1)$ or $(i',j')=(i+1,j)$. Here
  we assume that a diagonal cell is in every page. 
\end{itemize}

Observe that $\YB(n,\vectr,\vects)$ and $\LL(P_{n,\vectr,\vects})$ are
essentially the same: $B\in\YB(n,\vectr,\vects)$ corresponds to
$\pi\in \LL(P_{n,\vectr,\vects})$ such that for a cell $c$, the label
of $c$ in $B$ is $\pi(c)$.

Now we fix the linear extension $\omega:P_{n,\vectr,\vects}\to [N]$
defined uniquely by the following rules. Let $c$ and $c'$ be two cells
in rows $i$ and $i'$, columns $j$ and $j'$, and pages $k$ and $k'$,
respectively, where we define $k=0$ if $c$ is a diagonal cell. Then we
have $\omega(c)<\omega(c')$ if and only if one of the following
conditions holds: (1) $i<i'$ or (2) $i=i'$ and $k<k'$ or (3) $i=i'$,
$k=k'$, and $j<j'$, see Figure~\ref{fig:natural}.

\begin{figure}
  \centering
  \includegraphics{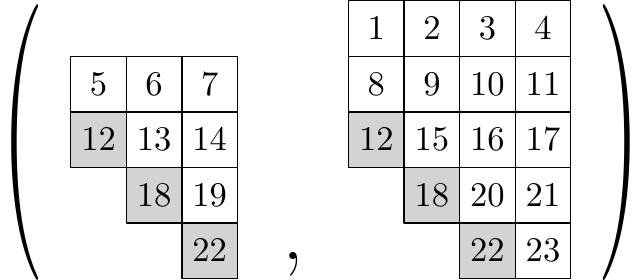}  
  \caption{The $(n,\vectr,\vects)$-Young book corresponding to the
    fixed linear extension $\omega$ of the poset $P_{n,\vectr,\vects}$
    with $n=3,\vectr=(1,2)$, and $\vects=(0,1)$.}
  \label{fig:natural}
\end{figure}


Note that if $B\in\YB(n,\vectr,\vects)$ corresponds to
$\pi\in\LL(P_{n,\vectr,\vects})$, the descents of $B$ and $\omega
\pi^{-1}$ are the same. Thus we have 
\begin{equation}
  \label{eq:YBL}
\sum_{B\in \YB(n,\vectr,\vects)} q^{\maj(B)}
 =
\sum_{\pi\in \LL(P_{n,\vectr,\vects})} q^{\maj(\omega \pi^{-1})}.
\end{equation}

By \eqref{eq:maj} and \eqref{eq:YBL}, we obtain
\begin{equation}
  \label{eq:AP}
\sum_{\sigma \in \A(P_{n,\vectr,\vects})} q^{|\sigma|}
=  \frac{\sum_{B\in \YB(n,\vectr,\vects)} q^{\maj(B)}}{(q;q)_N}.
\end{equation}

The {\em profile} of a $P_{n,\vectr,\vects}$-partition $\sigma \in \A(P_{n,\vectr,\vects})$ is defined to be 
the partition obtained by reading the diagonal entries of $\sigma$.
For a partition $\lambda$ of length $\le n$, 
let $\A_\lambda(P_{n,\vectr,\vects})$ be the set of all $P_{n,\vectr,\vects}$-partitions 
with profile $\lambda$.
Then we have
\[
\sum_{\sigma \in \A(P_{n,\vectr,\vects})} q^{|\sigma|}
 =
\sum_{\lambda \in \Par_n}
 \sum_{\sigma \in \A_\lambda(P_{n,\vectr,\vects})} q^{|\sigma|}.
\]

First we consider the case where $m=1$.
If a $P_{n+s,r,0}$-partition has the profile $(\lambda_1, \dots, \lambda_n, 0, \dots, 0)$, 
where $\lambda$ is a partition of length $\le n$,
then all the entries in the bottom triangle are $0$, 
and the remaining entries form a $P_{n,r,s}$-partition with profile $\lambda$.
So we have
\[
\sum_{\sigma \in \A_{(\lambda_1, \dots, \lambda_n)}(P_{n,r,s})} q^{|\sigma|}
 =
\sum_{\sigma \in \A_{(\lambda_1, \dots, \lambda_n, 0, \dots, 0)}(P_{n+s,r,0})} q^{|\sigma|}.
\]
In general, the generating function of $P_{l,r,0}$-partitions with fixed profile 
is expressed in terms of Schur function.
The following proposition is a special case of \cite[Theorem~2.1]{Okada}.
For the sake of completeness, we repeat the proof in our situation.

\begin{prop}
\label{prop:P-par-1}
For a partition $\mu$ of length $\le l$, we have
\begin{equation}
\label{eq:P-par-1}
\sum_{\sigma \in \A_\mu(P_{l,r,0})} q^{|\sigma|}
 =
\frac{ \prod_{h=1}^{r-1} (q;q)_h }
     { \prod_{h=l}^{l+r-1} (q;q)_h }
s_\mu(q^{r+1}, q^{r+2}, \dots, q^{r+l}).
\end{equation}
\end{prop}

\begin{proof}
  For $\sigma \in \A_\mu(P_{l,r,0})$, let $\sigma^{(i)}$ be the
  partition obtained by reading the entries in the cells in row $k$
  and column $k+i$ for some $k$ from northwest to southeast.  Then we
  have
\[
\mu = \sigma^{(0)} \prec \sigma^{(1)} \prec \sigma^{(2)} \prec
 \dots \prec \sigma^{(r-1)} \succ \sigma^{(r)}\succ \sigma^{(r+1)} \succ \dots \succ \sigma^{(l+r)} = \emptyset,
\]
where $\alpha \succ \beta$ means that the skew diagram $\alpha/\beta$ is 
a horizontal strip, i.e.,
\[
\alpha_1 \ge \beta_1 \ge \alpha_2 \ge \beta_2 \ge \dots.
\]

Let $\Lambda$ be the ring of symmetric functions, 
and define linear operators $H(t)$, $H^\perp(t)$, and $D(q)$ on $\Lambda$ by
\[
H(t) s_\alpha
 = \sum_{\gamma \succ \alpha} t^{|\gamma|-|\alpha|} s_\gamma,
\quad
H^\perp(t) s_\alpha
 = \sum_{\beta \prec \alpha} t^{|\alpha|-|\beta|} s_\beta,
\quad
D(q) s_\alpha
 = q^{|\alpha|} s_\alpha.
\]
Note that $H(t)$ is the multiplication by $\sum_{k \ge 0} h_k t^k$.
If we put
\[
Z_\mu(q)
 =
\sum_{\sigma \in \A_\mu(P_{l,r,0})} q^{|\sigma|},
\]
then we have
\[
\sum_{\mu \in \Par_l} Z_\mu(q) s_\mu
 =
\left( D(q) H^\perp(1) \right)^r
\left( D(q) H(1) \right)^l
s_{\emptyset}.
\]
By using the relations
\[
D(q) D(q') = D(qq'),
\quad
D(q) H(t) = H(qt) D(q),
\quad\text{and}\quad
D(q) H^\perp(t) = H^\perp(q^{-1} t) D(q),
\]
we see that
\begin{multline*}
\left( D(q) H^\perp(1) \right)^r
\left( D(q) H(1) \right)^l
\\
=
H^\perp(q^{-1}) H^\perp(q^{-2}) \cdots H^\perp(q^{-r})
H(q^{r+1}) H(q^{r+2}) \cdots H(q^{r+l})
D(q^{r+l}).
\end{multline*}
Next, by applying the commuting relation \cite[I.5 Example~29]{Macdonald}
\[
H^\perp(s) H(t)
 =
\frac{1}{1 - st} H(t) H^\perp(s),
\]
we obtain
\begin{multline*}
H^\perp(q^{-1}) H^\perp(q^{-2}) \cdots H^\perp(q^{-r})
H(q^{r+1}) H(q^{r+2}) \cdots H(q^{r+l})
\\
=
\prod_{i=1}^r \prod_{j=1}^l \frac{1}{1 - q^{(-i)+(r+j)}}
H(q^{r+1}) H(q^{r+2}) \cdots H(q^{r+l})
H^\perp(q^{-1}) H^\perp(q^{-2}) \cdots H^\perp(q^{-r}).
\end{multline*}
Since $D(q) s_\emptyset = s_\emptyset$ and $H^\perp(t) s_\emptyset = s_\emptyset$, 
we have
\[
\sum_{\mu \in \Par_l} Z_\mu(q) s_\mu
 =
\prod_{i=1}^r \prod_{j=1}^l \frac{1}{1 - q^{(-i)+(r+j)}}
H(q^{r+1}) H(q^{r+2}) \cdots H(q^{r+l})
s_\emptyset.
\]
Lastly, by appealing to the Cauchy identity
\[
H(x_1) H(x_2) \cdots H(x_l) s_\emptyset
 =
\sum_{\mu \in \Par_l} s_\mu(x_1, \dots, x_l) s_\mu,
\]
we conclude that
\[
\sum_{\mu \in \Par_l} Z_\mu(q) s_\mu
 =
\frac{ \prod_{h=1}^{r-1} (q;q)_h }
     { \prod_{h=l}^{l+r-1} (q;q)_h }
\sum_{\mu \in \Par_l} s_\mu(q^{r+1}, \dots, q^{r+l}) s_\mu.
\]
Equating the coefficient of $s_\mu$ completes the proof.
\end{proof}

Now we can compute the generating function of $P_{n,\vectr,\vects}$-partitions.

\begin{prop}
\label{prop:P-par-m}
Let $\vectr =(r_1, \dots, r_m)$ and $\vects = (s_1, \dots, s_m)$ be 
compositions of $r$ and $s$ respectively.
Then we have
\begin{equation}
\label{eq:P-par-m}
\sum_{\sigma \in \A(P_{n,\vectr,\vects})} q^{|\sigma|}
 =
\prod_{k=1}^m
 \frac{ \prod_{h=1}^{r_k-1} (q;q)_h }
      { \prod_{h=n+s_k}^{n=r_k+s_k-1} (q;q)_h }
\sum_{\lambda \in \Par_n}
 q^{|\lambda|}
 \prod_{k=1}^m q^{r_k |\lambda|} s_\lambda(1, q, \dots, q^{n+s_k-1}).
\end{equation}
\end{prop}

\begin{proof}
For a $P_{n,\vectr,\vects}$-partition $\sigma$ with profile $\lambda$, 
let $\sigma_k$ be the $P_{n,r_k,s_k}$-partition on the $k$th page.
Then we have
\[
|\sigma| = |\lambda| + \sum_{k=1}^m ( |\sigma_k| - |\lambda| ).
\]
Hence we have
\[
\sum_{\sigma \in \A(P_{n,\vectr,\vects})} q^{|\sigma|}
 =
\sum_{\lambda \in \Par_n}
 q^{|\lambda|}
 \prod_{k=1}^m 
  \sum_{\sigma_k \in \A_\lambda(P_{n,r_k,s_k})} q^{|\sigma_k| - |\lambda|}.
\]
On the other hand, 
it follows from Proposition~\ref{prop:P-par-1} and the homogeneity of Schur function that
\begin{align*}
\sum_{\sigma_k \in \A_\lambda(P_{n,r_k,s_k})} q^{|\sigma_k| - |\lambda|}
 &=
q^{-|\lambda|}
\frac{ \prod_{h=1}^{r_k-1} (q;q)_h }
     { \prod_{h=n+s_k}^{n+r_k+s_k-1} (q;q)_h }
s_\lambda(q^{r_k+1}, q^{r_k+2}, \dots, q^{n+r_k+s_k})
\\
 &=
q^{r_k |\lambda|}
\frac{ \prod_{h=1}^{r_k-1} (q;q)_h }
     { \prod_{h=n+s_k}^{n+r_k+s_k-1} (q;q)_h }
s_\lambda(1, q, \dots, q^{n+s_k-1}).
\end{align*}
\end{proof}

\section{%
Schur functions and $q$-Selberg integral
}
\label{sec:-schur-functions}

In this section, we prove Theorem~\ref{thm:qSelberg-Schur}, from which Theorem~\ref{thm:qKO} follows.

We begin with a general formula for the Jackson integral.

\begin{prop}
\label{prop:general}
Let $f$ be a function in $x_1, \dots, x_n$ satisfying the following two conditions:
\begin{enumerate}
\item[(a)]
$f$ is symmetric in $x_1, \dots, x_n$.
\item[(b)]
If $x_i = x_j$ for $i \neq j$, then $f(x_1, \dots, x_n) = 0$.
\end{enumerate}
Then we have
\begin{equation}
\label{eq:general}
\int_{[0,1]^n} f(x) d_qx
 = 
n! (1-q)^n
\sum_{\lambda \in \Par_n}
 q^{|\lambda|+\binom{n}{2}}
 f(q^{\lambda_1+n-1}, q^{\lambda_2+n-2}, \dots, q^{\lambda_n}).
\end{equation}
\end{prop}

\begin{proof}
We put
\[
S = \int_{[0,1]^n} f(x) d_qx.
\]
By definition, we have
\[
S = (1-q)^n \sum_{k_1, \dots, k_n \ge 0} f(q^{k_1}, \dots, q^{k_n}) q^{k_1 + \dots + k_n}.
\]
Condition (b) implies that $f(q^{k_1}, \dots, q^{k_n}) = 0$ if $k_i = k_j$ for some $i \neq j$.
Hence we may assume that $k_1, \dots, k_n$ are distinct in the summation.
For a permutation $\sigma \in \Sym_n$, we define
\[
S_\sigma
=
(1-q)^n \sum_{k_{\sigma(1)} > \dots > k_{\sigma(n)} \ge 0} f(q^{k_1}, \dots, q^{k_n}) q^{k_1 + \dots + k_n}.
\]
Then we have
\[
S = \sum_{\sigma \in \Sym_n} S_\sigma,
\]
and Condition (a) implies that
\[
S_\sigma = S_e \quad(\sigma \in \Sym_n),
\]
where $e$ is the identity permutation.
Therefore we have
\[
S
 = n! S_e
 = n! (1-q)^n \sum_{k_1 > \dots > k_n \ge 0} f(q^{k_1}, \dots, q^{k_n}) q^{k_1 + \dots + k_n}.
\]
For $k_1 > \dots > k_n \ge 0$, let $\lambda=(\lambda_1,\dots,\lambda_n)$ be
the partition given by $k_i=\lambda_i+(n-i)$. Then we have
\[
S
 =
n! (1-q)^n \sum_{\lambda \in \Par_n}
 q^{|\lambda| + \binom{n}{2}} f(q^{\lambda_1+n-1}, \dots, q^{\lambda_n}).
\]
\end{proof}

The principal specialization of Schur function can be written in the following form.

\begin{lem}
\label{lem:q-Schur}
Let $n$ be a positive integer and $\lambda$ a partition of length $\le n$.
If $s$ is a nonnegative integer, then we have
\begin{equation}
\label{eq:q-Schur}
s_\lambda(1, q, \dots, q^{n+s-1})
 =
\frac{q^{ - \binom{n}{3} }}{\prod_{h=s}^{n+s-1} (q;q)_h}
\prod_{i=1}^n ( q \cdot q^{\lambda_i+n-i} ; q )_s
\prod_{1 \le i < j \le n} (q^{\lambda_j+n-j} - q^{\lambda_i+n-i}).
\end{equation}
\end{lem}

\begin{proof}
Using the Vandermonde determinant, we have
\[
s_\lambda(1,q,\dots,q^{n+s-1})
 =
\frac{ \prod_{1 \le i < j \le n+s} \left( q^{\lambda_j+n+s-j} - q^{\lambda_i+n+s-i} \right) }
     { \prod_{1 \le i < j \le n+s} \left( q^{n+s-j} - q^{n+s-i} \right) }.
\]
Since $\lambda_{n+1} = \dots = \lambda_{n+s} = 0$, we have
\begin{align*}
\prod_{1 \le i < j \le n} ( q^{\lambda_j+n+s-j} - q^{\lambda_i+n+s-i} )
&=
q^{s \binom{n}{2}}
\prod_{1 \le i < j \le n} \left( q^{\lambda_j+n-j} - q^{\lambda_i+n-i} \right),
\\
\prod_{i=1}^n \prod_{j=n+1}^{n+s} ( q^{\lambda_j+n+s-j} - q^{\lambda_i+n+s-i} )
&=
q^{n \binom{s}{2}}
\prod_{i=1}^n \left( q^{\lambda_i+n-i+1} ; q \right)_s,
\\
\prod_{n+1 \le i < j \le n+s} ( q^{\lambda_j+n+s-j} - q^{\lambda_i+n+s-i} )
&=
q^{\binom{s}{3}}
\prod_{h=1}^{s-1} (q;q)_h.
\end{align*}
On the other hand, the denominator is equal to
\[
\prod_{1 \le i < j \le n+s} ( q^{n+s-j} - q^{n+s-i} )
 =
q^{\binom{n+s}{3}}
\prod_{h=1}^{n+s-1} (q;q)_h.
\]
We can complete the proof by noting the relation
\[
s \binom{n}{2} + n \binom{s}{2} + \binom{s}{3} - \binom{n+s}{3}
 =
- \binom{n}{3}.
\]
\end{proof}

Now we can prove Theorem~\ref{thm:qSelberg-Schur}.

\begin{proof}[Proof of Theorem~\ref{thm:qSelberg-Schur}]
We put
\[
f_{n,r,s}(x_1, \dots, x_n)
 =
\prod_{i=1}^n x_i^r (q x_i ; q)_s
\prod_{1 \le i < j \le n} |x_j - x_i|.
\]
Then, by using Proposition~\ref{prop:general}, we have
\[
\frac{1}{n!}
\int_{[0,1]^n} \prod_{k=1}^m f_{n,r_k,s_k}(x_1, \dots, x_n) d_q x
 =
(1-q)^n
\sum_{\lambda\in\Par_n}
 q^{|\lambda| + \binom{n}{2}}
 \prod_{k=1}^m f(q^{\lambda_1+n-1}, \dots, q^{\lambda_n}).
\]
It follows from Lemma~\ref{lem:q-Schur} that
\begin{align*}
f_{n,r,s}(q^{\lambda_1+n-1}, \dots, q^{\lambda_n})
 &=
\prod_{i=1}^n
 \left( q^{\lambda_i+n-i} \right)^r
 (q \cdot q^{\lambda_i + n-i} ; q)_s
\prod_{1 \le i < j \le n}
 |q^{\lambda_j+n-j} - q^{\lambda_i+n-i}|
\\
 &=
q^{ r|\lambda| + r \binom{n}{2} }
\prod_{i=1}^n
 (q \cdot q^{\lambda_i + n-i} ; q)_s
\prod_{1 \le i < j \le n}
 (q^{\lambda_j+n-j} - q^{\lambda_i+n-i})
\\
 &=
q^{ r|\lambda| + r \binom{n}{2} + \binom{n}{3} }
\prod_{h=s}^{n+s-1} (q;q)_h 
s_\lambda(1, q, \dots, q^{n+s-1}).
\end{align*}
Hence we have
\begin{align*}
&
\frac{1}{n!}
\int_{[0,1]^n} \prod_{k=1}^m f_{n,r_k,s_k}(x_1, \dots, x_n) d_q x
\\
&=
(1-q)^n
\sum_{\lambda\in\Par_n}
 q^{|\lambda| + \binom{n}{2}}
\prod_{k=1}^m
 q^{ r_k|\lambda| + r_k \binom{n}{2} + \binom{n}{3} }
\prod_{h=s_k}^{n+s_k-1} (q;q)_h 
 s_\lambda(1, q, \dots, q^{n+s_k-1})
\\
&=
q^{ (r+1) \binom{n}{2} + m \binom{n}{3} }
(1-q)^n
\prod_{k=1}^m
\prod_{h=s_k}^{n+s_k-1} (q;q)_h 
\sum_{\lambda\in\Par_n}
 q^{|\lambda|}
 \prod_{k=1}^m
  q^{r_k |\lambda|}
  s_\lambda(1, q, \dots, q^{n+s_k-1}).
\end{align*}
\end{proof}

By Proposition~\ref{prop:P-par-m}, Theorem~\ref{thm:qSelberg-Schur}
and \eqref{eq:maj}, we obtain Theorem~\ref{thm:qKO}.


\section{%
Evaluation of $q$-Selberg integrals
}
\label{sec:-evaluation-q}

In this section, we use the Cauchy and Schur--Littlewood identities 
to evaluate some special cases of the $q$-Selberg integral in Theorem~\ref{thm:qKO}.

Recall the Cauchy and Schur--Littlewood identities:
\begin{gather}
\label{eq:Cauchy}
\sum_{\lambda \in \Par} s_{\lambda}(x) s_{\lambda}(y)
 = 
\frac{1}{\prod_{i,j} (1-x_i y_j)},
\\
\label{eq:Littlewood}
\sum_{\lambda \in \Par} s_{\lambda}(x) = 
\frac{1}{\prod_{i} (1-x_i) \prod_{i<j} (1-x_ix_j)},
\end{gather}
where $\Par$ denotes the set of all partitions.

\begin{thm}
\label{thm:eval-qSelberg}
For $0 < q < 1$, we have
\begin{multline}
\label{eq:eval-qSelberg1}
\int_{0}^1 \cdots \int_{0}^1
 \prod_{i=1}^n x_i^r 
 \prod_{1 \le i< j \le n} \left| x_j - x_i \right|
 d_q x_1 \cdots d_q x_n
\\
=
n! q^{ (r+1) \binom{n}{2} + \binom{n}{3} }
\frac{ [1]_q! [2]_q! \cdots [n-1]_q! }
     { \prod_{i=1}^n [r+i]_q 
       \prod_{1 \le i < j \le n} [2 r + i + j]_q },
\end{multline}
\begin{multline}
\label{eq:eval-qSelberg2}
\int_{0}^1 \cdots \int_{0}^1
 \prod_{i=1}^n x_i^r (1 - q x_i)
 \prod_{1 \le i< j \le n} \left| x_j - x_i \right|
 d_q x_1 \cdots d_q x_n
\\
=
n! q^{ (r+1) \binom{n}{2} + \binom{n}{3} }
\frac{ [1]_q! [2]_q! \cdots [n]_q! \left[ (n+1)r + \binom{n+1}{2} \right]_q }
     { \prod_{i=1}^{n+1} [r+i]_q 
       \prod_{1 \le i < j \le n+1} [2 r + i + j]_q },
\end{multline}
\begin{multline}
\label{eq:eval-qSelberg3}
\int_{0}^1 \cdots \int_{0}^1
\prod_{i=1}^n x_i^r (qx_i ; q)_s
\prod_{1 \le i < j \le n} \left| x_j - x_i \right|^2
d_q x_1 \cdots d_q x_n
\\
=
n! q^{ (r+1) \binom{n}{2} + 2 \binom{n}{3} }
\frac{ [1]_q! [2]_q! \cdots [n-1]_q! [s+1]_q! [s+2]_q! \cdots [s+n-1]_q! }
     { \prod_{i=1}^{n+s} \prod_{j=1}^n [r + i + j -1]_q}.
\end{multline}
\end{thm}

\begin{remark}
The last formula \eqref{eq:eval-qSelberg3} is equivalent to the $k=1$ case of \eqref{eq:Askey}.
See \cite[Section~4]{Kadell1}.
\end{remark}

\begin{proof}
By applying \eqref{eq:qSelberg} with $s=0$ and $m=1$ and 
using the homogeneity of Schur function, we have
\begin{multline*}
\int_{0}^1 \cdots \int_{0}^1
 \prod_{i=1}^n x_i^r 
 \prod_{1 \le i< j \le n} \left| x_j - x_i \right|
 d_q x_1 \cdots d_q x_n
\\
=
n! (1-q)^n
q^{ (r+1) \binom{n}{2} + \binom{n}{3} }
\prod_{h=1}^{n-1} (q;q)_h
\sum_{\lambda \in \Par_n}
 s_\lambda(q^{r+1}, q^{r+2}, \dots, q^{r+n}).
\end{multline*}
Thus \eqref{eq:eval-qSelberg1} is obtained by specializing 
$x_i = q^{r+i}$ for $1 \le i \le n$ and $x_i = 0$ for $i > n$ 
in the Schur--Littlewood identity \eqref{eq:Littlewood}.

Next we apply \eqref{eq:qSelberg} with $s=1$ and $m=1$. 
Then we have
\begin{multline*}
\int_{0}^1 \cdots \int_{0}^1
 \prod_{i=1}^n x_i^r (1 - q x_i)
 \prod_{1 \le i< j \le n} \left| x_j - x_i \right|
 d_q x_1 \cdots d_q x_n
\\
=
n! (1-q)^n
q^{ (r+1) \binom{n}{2} + \binom{n}{3} }
\prod_{h=1}^n (q;q)_h
\sum_{\lambda \in \Par_n}
 s_\lambda(q^{r+1}, q^{r+2}, \dots, q^{r+n+1}).
\end{multline*}
For $\lambda \in \Par_{n+1}$ and $\lambda+(1^{n+1}) = (\lambda_1+1, \dots, \lambda_{n+1}+1)$, we have
\[
s_{\lambda+(1^{n+1})}(x_1, \dots, x_{n+1})
 = x_1 \cdots x_{n+1} s_\lambda(x_1, \dots, x_{n+1}).
\]
Hence, by using \eqref{eq:Littlewood}, we see that
\begin{align*}
\sum_{\lambda \in \Par_n}
 s_\lambda(x_1, \dots, x_{n+1})
&=
\sum_{\lambda \in \Par_{n+1}} s_\lambda(x_1, \dots, x_{n+1})
-
\sum_{\lambda \in \Par_{n+1}, \lambda_{n+1} > 0} s_\lambda(x_1, \dots, x_{n+1})
\\
&=
\frac{ 1 - x_1 \cdots x_{n+1} }
     { \prod_{i=1}^{n+1} (1 - x_i)
       \prod_{1 \le i < j \le n+1} (1 - x_i x_j) }.
\end{align*}
Thus \eqref{eq:eval-qSelberg2} is obtained by specializing $x_i = q^{r+i}$ 
for $1 \le i \le n+1$.

Finally we use \eqref{eq:qSelberg} in the case $m=2$.
Then we have
\begin{multline*}
\int_{0}^1 \cdots \int_{0}^1
\prod_{i=1}^n x_i^r (qx_i ; q)_s
\prod_{1 \le i < j \le n} \left| x_j - x_i \right|^2
d_q x_1 \cdots d_q x_n
\\
=
n! (1-q)^n
q^{ (r+1) \binom{n}{2} + 2 \binom{n}{3} }
\prod_{h=s}^{n+s-1} (q;q)_h
\prod_{h=1}^{n-1} (q;q)_h
\sum_{\lambda \in \Par_n}
 s_\lambda(1, q, \dots, q^{n+s-1})
 s_\lambda(q^{r+1}, q^{r+2}, \dots, q^{r+n}).
\end{multline*}
Hence \eqref{eq:eval-qSelberg3} follows from the Cauchy identity \eqref{eq:Cauchy}.
\end{proof}

\begin{remark}
We have product formulas for the $q$-Selberg integral in Theorem~\ref{thm:qKO} 
in the case $m=1$ and $s \in \{ 0, 1 \}$.
If $m=1$ and $s \ge 2$, then we can evaluate the $q$-Selberg integral by using 
a column-length restricted Schur--Littlewood formula 
\cite{King2}
\[
\sum_{\lambda \in \Par_n} s_\lambda(x)
 =
\frac{1}{ \prod_i (1 - x_i) \prod_{i<j} (1 - x_i x_j)}
\cdot
\sum_\mu (-1)^{(|\mu|-(n-1)p(\mu))/2} s_\mu,
\]
where $\mu$ runs over all partitions of the form $\mu =
(\alpha|\alpha+n)$ in the Frobinius notation, and $p(\mu)$ is the 
largest integer $k$ such that $\mu$ contains a $k \times k$ square. 
However we do not have simple product formulas for $s \ge 2$.
\end{remark}

\section{%
Variants of the $q$-Selberg integral
}
\label{sec:-variants-q}

Under the same idea as in the previous section, 
we use the Cauchy-type identities for classical group characters 
(see \cite{King1} and \cite{Koike})
to derive variants of the $q$-Selberg integral \eqref{eq:Askey}.

\begin{thm}
\label{thm:variant}
Let $n$ be a positive integer and $s$ be a nonnegative integer.
Then we have
\begin{multline}
\label{eq:variant1}
\int_{[0,1]^n}
 \prod_{i=1}^n x_i^r (qx_i;q)_s
 \prod_{1 \le i < j \le n} (1 - q^{s+1} x_i x_j)
 \prod_{1 \le i < j \le n} |x_j - x_i|^2
d_qx
\\
=
n! q^{ (r+1) \binom{n}{2} + 2 \binom{n}{3} }
\prod_{k=1}^{n-1} [k]_q!
\prod_{k=1}^n [s+2k-2]_q!
\frac{ \prod_{1 \le i < j \le n} [2n+2r+s+i+j-2]_q }
     { \prod_{i=1}^{2n+s-1} \prod_{j=1}^n [r+i+j-1]_q },
\end{multline}
\begin{multline}
\label{eq:variant2}
\int_{[0,1]^n}
 \prod_{i=1}^n x_i^r (qx_i;q)_s ( 1 + q^{(s+1)/2} x_i )
 \prod_{1 \le i < j \le n} (1 - q^{s+1} x_i x_j)
 \prod_{1 \le i < j \le n} |x_j - x_i|^2
d_qx
\\
=
n! q^{ (r+1) \binom{n}{2} + 2 \binom{n}{3} }
\prod_{k=1}^{n-1} [k]_q!
\prod_{k=1}^n [s+2k-2]_q!
\prod_{k=1}^n (1+q^{s/2+k-1/2})
\\
\times
\frac{ \prod_{i=1}^n [n+r+s/2+i-1/2]_q
       \prod_{1 \le i < j \le n} [2n+2r+s+i+j-1]_q }
     { \prod_{i=1}^{2n+s} \prod_{j=1}^n [r+i+j-1]_q },
\end{multline}
\begin{multline}
\label{eq:variant3}
\int_{[0,1]^n}
 \prod_{i=1}^n x_i^r (qx_i;q)_s (1 - q^{(s+1)/2} x_i)
 \prod_{1 \le i < j \le n} (1 - q^{s+1} x_i x_j)
 \prod_{1 \le i < j \le n} |x_j - x_i|^2
d_qx
\\
=
n! q^{ (r+1) \binom{n}{2} + 2 \binom{n}{3} }
\prod_{k=1}^{n-1} [k]_q!
\prod_{k=1}^n [s+2k-2]!
\prod_{k=1}^n [s/2+k-1/2]_q
\\
\times
\frac{ \prod_{i=1}^n (1 + q^{n+r+s/2+i-1/2})
       \prod_{1 \le i < j \le n} [2n+2r+s+i+j-1]_q }
     { \prod_{i=1}^{2n+s} \prod_{j=1}^n [r+i+j-1]_q },
\end{multline}
\begin{multline}
\label{eq:variant4}
\int_{[0,1]^n}
 \prod_{i=1}^n x_i^r (qx_i;q)_s
 \prod_{1 \le i \le j \le n} (1 - q^{s+1} x_i x_j)
 \prod_{1 \le i < j \le n} |x_j - x_i|^2
d_qx
\\
=
n! q^{ (r+1) \binom{n}{2} + 2 \binom{n}{3} }
\prod_{k=1}^{n-1} [k]_q!
\prod_{k=1}^n [s+2k-1]_q!
\frac{ \prod_{1 \le i \le j \le n} [2n+2r+s+i+j]_q }
     { \prod_{i=1}^{2n+s+1} \prod_{j=1}^n [r+i+j-1]_q }.
\end{multline}
\end{thm}

In order to prove this theorem, we use the classical group characters.
Let $\lambda \in \Par_N$ be a partition of length $\le N$.
The \emph{symplectic Schur function} $s^C_{\langle \lambda \rangle}(x_1, \dots, x_N)$, 
and the \emph{orthogonal Schur function} 
$s^B_{[\lambda]}(x_1, \dots, x_N)$, $s^D_{[\lambda]}(x_1, \dots, x_N)$ are defined by putting
\begin{align}
s^C_{\langle \lambda \rangle}(x_1, \dots, x_N)
 &=
\frac{ \det \left( x_i^{\lambda_j+N-j+1} - x_i^{-(\lambda_j+N-j+1)} \right)_{1 \le i, j \le N} }
     { \det \left( x_i^{N-j+1} - x_i^{-(N-j+1)} \right)_{1 \le i, j \le N} },
\label{eq:c-schur}
\\
s^B_{[ \lambda ]}(x_1, \dots, x_N)
 &=
\frac{ \det \left( x_i^{\lambda_j+N-j+1/2} - x_i^{-(\lambda_j+N-j+1/2)} \right)_{1 \le i, j \le N} }
     { \det \left( x_i^{N-j+1/2} - x_i^{-(N-j+1/2)} \right)_{1 \le i, j \le N} },
\label{eq:b-schur}
\\
s^D_{\langle \lambda \rangle}(x_1, \dots, x_N)
 &=
\begin{cases}
\dfrac{ \det \left( x_i^{\lambda_j+N-j} + x_i^{-(\lambda_j+N-j)} \right)_{1 \le i, j \le N} }
      { \dfrac{1}{2} \det \left( x_i^{N-j} + x_i^{-(N-j)} \right)_{1 \le i, j \le N} }
 &\text{if $\lambda_N > 0$,} \\
\dfrac{ \det \left( x_i^{\lambda_j+N-j} + x_i^{-(\lambda_j+N-j)} \right)_{1 \le i, j \le N} }
      { \det \left( x_i^{N-j} + x_i^{-(N-j)} \right)_{1 \le i, j \le N} }
 &\text{if $\lambda_N = 0$.}
\end{cases}
\label{eq:d-schur}
\end{align}
For a decreasing sequence of the form $\lambda+1/2 
= (\lambda_1+1/2, \dots, \lambda_N+1/2)$, where $\lambda \in \Par_N$, 
we define the \emph{spinor Schur function} $s^B_{[\lambda+1/2]}(x_1, \dots, x_N)$ and 
$s^D_{[\lambda+1/2]}(x_1, \dots, x_N)$ by putting
\begin{align}
s^B_{[ \lambda + 1/2 ]}(x_1, \dots, x_N)
 &=
\frac{ \det \left( x_i^{\lambda_j+N-j+1} - x_i^{-(\lambda_j+N-j+1)} \right)_{1 \le i, j \le N} }
     { \det \left( x_i^{N-j+1/2} - x_i^{-(N-j+1/2)} \right)_{1 \le i, j \le N} },
\label{eq:bs-schur1}
\\
s^D_{ [ \lambda + 1/2 ] }(x_1, \dots, x_N)
 &=
\frac{ \det \left( x_i^{\lambda_j+N-j+1/2} + x_i^{-(\lambda_j+N-j+1/2)} \right)_{1 \le i, j \le N} }
     { \dfrac{1}{2} \det \left( x_i^{N-j} + x_i^{-(N-j)} \right)_{1 \le i, j \le N} }.
\label{eq:ds-schur}
\end{align}
These schur functions are irreducible characters of classical groups.
Note that
$$
s^B_{[\lambda+1/2]}(x_1, \dots, x_N)
 =
\prod_{i=1}^N (x_i^{1/2} + x_i^{-1/2})
s^C_{\langle \lambda \rangle}(x_1, \dots, x_N).
$$
Then we have the following Cauchy-type identities, 
which can be proved by using the Cauchy--Binet formula.

\begin{prop}
\label{prop:cauchy}
Let $N$ and $n$ be positive integers satisfying $N \ge n$.
Then we have
\begin{align}
\sum_{\lambda \in \Par_n}
 s^C_{\langle \lambda \rangle}(x_1, \dots, x_N)
 s_\lambda(u_1, \dots, u_n)
&=
\frac{ \prod_{1 \le i < j \le n} (1 - u_i u_j) }
     { \prod_{i=1}^N \prod_{j=1}^n (1 - x_i u_j) (1 - x_i^{-1} u_j) },
\label{eq:c-cauchy}
\\
\sum_{\lambda \in \Par_n}
 s^B_{ [ \lambda ] }(x_1, \dots, x_N)
 s_\lambda(u_1, \dots, u_n)
&=
\frac{ \prod_{i=1}^n (1 + u_i) \prod_{1 \le i < j \le n} (1 - u_i u_j) }
     { \prod_{i=1}^N \prod_{j=1}^n (1 - x_i u_j) (1 - x_i^{-1} u_j) },
\label{eq:b-cauchy}
\\
\sum_{\lambda \in \Par_n}
 s^B_{ [ \lambda +1/2 ] }(x_1, \dots, x_N)
 s_\lambda(u_1, \dots, u_n)
&=
\frac{ \prod_{i=1}^n (x_i^{1/2} + x_i^{-1/2} ) \prod_{1 \le i < j \le n} (1 - u_i u_j) }
     { \prod_{i=1}^N \prod_{j=1}^n (1 - x_i u_j) (1 - x_i^{-1} u_j) },
\label{eq:bs-cauchy}
\\
\sum_{\lambda \in \Par_n}
 s^D_{ [ \lambda ] }(x_1, \dots, x_N)
 s_\lambda(u_1, \dots, u_n)
&=
\frac{ \prod_{1 \le i \le j \le n} (1 - u_i u_j) }
     { \prod_{i=1}^N \prod_{j=1}^n (1 - x_i u_j) (1 - x_i^{-1} u_j) },
\label{eq:d-cauchy}
\\
\sum_{\lambda \in \Par_n}
 s^D_{ [ \lambda + 1/2 ] }(x_1, \dots, x_N)
 s_\lambda(u_1, \dots, u_n)
&=
\frac{ \prod_{i=1}^n (x_i^{1/2} + x_i^{-1/2})
       \prod_{i=1}^n (1 -u_i) \prod_{1 \le i < j \le n} (1 - u_i u_j) }
     { \prod_{i=1}^N \prod_{j=1}^n (1 - x_i u_j) (1 - x_i^{-1} u_j) }.
\label{eq:ds-cauchy}
\end{align}
\end{prop}

\begin{proof}[Proof of Theorem~\ref{thm:variant}]
Since the arguments are the same, we illustrate how to prove \eqref{eq:variant1} 
in the case $s$ is odd.
For other cases, see Table~\ref{table:proof}.

\begin{table}[htbp]
\caption{Proof of Theorem~\ref{thm:variant}}
\label{table:proof}
\centering
\begin{tabular}{c|c|c|c}
Identity            & Parity of $s$ & Character & Specialization \\
\hline
\eqref{eq:variant1} & odd           & $s^C_{\langle \lambda \rangle}$ & $x_i = q^{N-i+1/2}$ \\
\hline
\eqref{eq:variant1} & even           & $s^D_{[\lambda +1/2]}$          & $x_i = q^{N-i}$ \\
\hline
\eqref{eq:variant2} & odd           & $s^C_{\langle \lambda \rangle}$ & $x_i = q^{N-i+1}$ \\
\hline
\eqref{eq:variant2} & even          & $s^D_{[\lambda + 1/2]}$         & $x_i = q^{N-i+1/2}$ \\
\hline
\eqref{eq:variant3} & odd           & $s^D_{[\lambda]}$               & $x_i = q^{N-i}$ \\
\hline
\eqref{eq:variant3} & even          & $s^B_{[\lambda]}$               & $x_i = q^{N-i+1/2}$ \\
\hline
\eqref{eq:variant4} & odd           & $s^D_{[\lambda]}$               & $x_i = q^{N-i+1/2}$ \\
\hline
\eqref{eq:variant4} & even          & $s^B_{[\lambda]}$               & $x_i = q^{N-i+1}$ \\
\end{tabular}
\end{table}

Suppose that $s = 2l+1$ is odd and put $N = n+l$.
If follows from the Weyl denominator formula (or the Vandermonde determinant) that, 
for a partition $\lambda$ of length $\le N$, 
\begin{multline*}
s^C_{\langle \lambda \rangle}(q^{N-1/2}, q^{N-3/2}, \dots, q^{3/2}, q^{1/2})
\\
 =
q^{ - (N-1/2) |\lambda| }
\prod_{i=1}^N
 \frac{ 1 - q^{\lambda_i+N-i+1} }
      { 1 - q^{N-i+1} }
\prod_{1 \le i < j \le N}
 \frac{ q^{\lambda_j+N-j+1} - q^{\lambda_i+N-i+1} }
      { q^{N-j+1} - q^{N-i+1} }
 \frac{ 1 - q^{\lambda_i+N-i+1} \cdot q^{\lambda_j+N-j+1} }
      { 1 - q^{N-i+1} \cdot q^{N-j+1} }.
\end{multline*}
If $l(\lambda) \le n$, then we have
\[
s^C_{\langle \lambda \rangle}(q^{N-1/2}, q^{N-3/2}, \dots, q^{3/2}, q^{1/2})
 =
q^{- (N-1/2) |\lambda| - \binom{n}{3}}
\frac{ 1 }
     { \prod_{k=1}^n (q;q)_{s+2k-2} }
\cdot
F(q^{\lambda_1+n-1}, \dots, q^{\lambda_n}),
\]
where
\[
F(x_1, \dots, x_n)
 =
\prod_{i=1}^n (qx_i;q)_s
\prod_{1 \le i < j \le n} (1 - q^{s+1} x_i x_j)
\prod_{1 \le i < j \le n} (x_j - x_i).
\]
Hence, by applying Proposition~\ref{prop:general}, we see that
\begin{align*}
&
\int_{[0,1]^n}
 \prod_{i=1}^n x_i^r (qx_i;q)_s
 \prod_{1 \le i < j \le n} (1 - q^{s+1} x_i x_j)
 \prod_{1 \le i < j \le n} |x_j - x_i|^2
d_qx
\\
&=
n! (1-q)^n
\sum_{\lambda \in \Par_n}
 q^{ |\lambda| + \binom{n}{2} }
 \prod_{i=1}^n \left( q^{\lambda_i+n-i} \right)^r
 F(q^{\lambda_1+n-1}, \dots, q^{\lambda_n})
 \prod_{1 \le i < j \le n} (q^{\lambda_j+n-j} - q^{\lambda_i+n-i})
\\
&=
n! (1-q)^n
q^{ (r+1) \binom{n}{2} + 2 \binom{n}{3} }
\prod_{k=1}^{n-1} [k]_q!
\prod_{k=1}^n [s+2k-2]_q!
\\
&\quad\times
\sum_{\lambda \in \Par_n}
 s^C_{\langle \lambda \rangle}(q^{N-1/2}, q^{N-3/2}, \dots, q^{3/2}, q^{1/2})
 \cdot
 q^{(r+N+1/2) |\lambda|} s_\lambda(1, q, \dots, q^{n-1})
\end{align*}
Now, by specializing $x_i = q^{N-i+1/2}$ and $u_j = q^{r+N+j-1/2}$ in 
the Cauchy-type identity \eqref{eq:c-cauchy}, 
we obtain \eqref{eq:variant1}.
\end{proof}

Similarly, by using the Cauchy-type formula for Schur functions corresponding 
to rational irreducible representations of the general linear group, 
we can derive

\begin{thm}
\label{thm:rat-variant}
Let $n$, $m$ and $l$ be nonnegative integers with $N = n+m+l > 0$.
Then we have
\begin{multline}
\label{eq:rat-variant}
\int_{[0,1]^{n+m}}
 \prod_{i=1}^n x_i^r (qx_i;q)_l
 \prod_{j=1}^m y_j^s (qy_j;q)_l
 \prod_{i=1}^n \prod_{j=1}^m (1 - q^l x_i y_j)
 \prod_{1 \le i < j \le n} (x_j - x_i)^2
 \prod_{1 \le i < j \le m} (y_j - y_i)^2
 d_qx d_qy
\\
 =
n! m! q^{(r+1) \binom{n}{2} + (s+1) \binom{m}{2} + 2 \binom{n}{3} + 2 \binom{m}{3}}
\frac{ \prod_{k=1}^{N-1} [k]_q! \prod_{k=1}^{n-1} [k]_q! \prod_{k=1}^{m-1} [k]_q! }
     { \prod_{k=1}^{l-1} [k]_q! }
\\
\times
\frac{ \prod_{i=1}^n \prod_{j=1}^m [N+r+s+i+j-1] }
     { \prod_{i=1}^n \prod_{k=1}^N [r+i+k-1]_q
       \prod_{j=1}^m \prod_{k=1}^N [s+j+k-1]_q }.
\end{multline}
\end{thm}

In order to prove this theorem, we use the generalized Schur function defined by
\[
s_\Lambda(x_1, \dots, x_N)
 =
\frac{ \det \left( x_i^{\Lambda_j+N-j} \right)_{1 \le i, j \le N} }
     { \det \left( x_i^{N-j} \right)_{1 \le i, j \le N} },
\]
where $\Lambda \in \Z^n$ is a decreasing sequence of integers.
For two partitions $\lambda$ and $\mu$ with $l(\lambda) + l(\mu) \le N$, 
we define a weakly decreasing sequence $\Lambda(\lambda,\mu) \in \Z^N$ by
\[
\Lambda(\lambda,\mu)
 =
(\lambda_1, \dots, \lambda_{l(\lambda)}, 0, \dots, 0, -\mu_{l(\mu)}, \dots, -\mu_1).
\]
Then we have the following Cauchy-type formula:

\begin{prop}
\label{prop:rat-cauchy}
If $N \ge n+m$, then we have
\begin{multline}
\label{eq:rat-cauchy}
\sum_{\lambda \in \Par_n, \mu \in \Par_m}
 s_{\Lambda(\lambda,\mu)}(x_1, \dots, x_N) 
 s_\lambda(u_1, \dots, u_n) s_\mu(v_1, \dots, v_m)
\\
=
\frac{ \prod_{i=1}^n \prod_{j=1}^m (1 - u_i v_j) }
     { \prod_{i=1}^N \prod_{j=1}^n (1 - x_i u_j) \prod_{j=1}^m (1 - x_i^{-1} v_j) }.
\end{multline}
\end{prop}

Also we need the following extension of Proposition~\ref{prop:general}.

\begin{prop}
\label{prop:general2}
Let $g$ be a function in $x_1, \dots, x_n, y_1, \dots, y_m$ satisfying 
the following four conditions:
\begin{enumerate}
\item[(a1)]
$g$ is symmetric in $x_1, \dots, x_n$;
\item[(a2)]
$g$ is symmetric in $y_1, \dots, y_m$;
\item[(b1)]
If $x_i = x_j$ for $i \neq j$, then $g(x_1, \dots, x_n, y_1, \dots, y_m) = 0$;
\item[(b2)]
If $y_i = y_j$ for $i \neq j$, then $g(x_1, \dots, x_n, y_1, \dots, y_m) = 0$.
\end{enumerate}
Then we have
\begin{multline}
\label{eq:general2}
\int_{[0,1]^{n+m}} g(x,y) d_qx d_qy
\\
 = 
(1-q)^{n+m} n! m!
\sum_{\lambda \in \Par_n, \mu \in \Par_m}
 q^{|\lambda| + |\mu| + \binom{n}{2} + \binom{m}{2}}
 g( q^{\lambda_1+n-1}, q^{\lambda_2+n-2}, \dots, q^{\lambda_n},
    q^{\mu_1+m-1}, q^{\mu_2+m-2}, \dots, q^{\mu_m} ).
\end{multline}
\end{prop}

\begin{proof}[Proof of Theorem~\ref{thm:rat-variant}]
The idea of the proof is similar to that of Theorem~\ref{thm:variant}.

If we put
\begin{multline*}
G(x_1, \dots, x_n, y_1, \dots, y_m)
\\
 =
 \prod_{i=1}^n (qx_i;q)_l
 \prod_{j=1}^m (qy_j;q)_l
 \prod_{i=1}^n \prod_{j=1}^m (1 - q^l x_i y_j)
 \prod_{1 \le i < j \le n} (x_j - x_i)
 \prod_{1 \le i < j \le m} (y_j - y_i),
\end{multline*}
then we have
\[
s_{\Lambda(\lambda,\mu)}(1, q, \dots, q^{n-1})
 =
q^{ -(N-1) |\mu| - \binom{n}{3} - \binom{m}{3} }
\frac{ \prod_{h=1}^{l-1} (q;q)_h }
     { \prod_{h=1}^{N-1} (q;q)_h }
\cdot
G(q^{\lambda_1+n-1}, \dots, q^{\lambda_n}).
\]
Hence, by applying Proposition~\ref{prop:general2}, we see that
\begin{multline*}
\int_{[0,1]^{n+m}}
 \prod_{i=1}^n x_i^r (qx_i;q)_l
 \prod_{j=1}^m y_j^s (qy_j;q)_l
 \prod_{i=1}^n \prod_{j=1}^m (1 - q^l x_i y_j)
 \prod_{1 \le i < j \le n} (x_j - x_i)^2
 \prod_{1 \le i < j \le m} (y_j - y_i)^2
 d_qx d_qy
\\
=
n! m! (1-q)^{n+m}
q^{ (r+1) \binom{n}{2} + (s+1) \binom{m}{2} + 2 \binom{n}{3} + 2 \binom{m}{3} }
\frac{ \prod_{h=1}^{N-1} (q;q)_h \prod_{h=1}^{n-1} (q;q)_h \prod_{h=1}^{m-1} (q;q)_h }
     { \prod_{h=1}^{l-1} (q;q)_h }
\\
\times
\sum_{\lambda \in \Par_n, \mu \in \Par_m}
 s_{\Lambda(\lambda,\mu)}(1, q, \dots, q^{N-1})
 \cdot q^{(r+1) |\lambda|} s_\lambda(1, q, \dots, q^{n-1})
 \cdot q^{(s+N) |\mu|} s_\mu(1, q, \dots ,q^{m-1}).
\end{multline*}
Now we can complete the proof by using the Cauchy-type identity \eqref{eq:rat-cauchy}.
\end{proof}


\end{document}